\documentclass{amsart}
\calclayout
\usepackage{amssymb,amsmath,amsfonts,epsfig,latexsym,tikz}
\usepackage{tikz-cd}
\usepackage{hyperref}
\usepackage{enumerate}
\usepackage{mathtools}
\usepackage{verbatim}
\usepackage{cleveref}
\usepackage[shortlabels]{enumitem}
\usepackage{subcaption}
\usepackage{stmaryrd}

\usetikzlibrary{positioning}
\usetikzlibrary{matrix}
\usetikzlibrary{decorations}
\usetikzlibrary{decorations.pathreplacing, decorations.pathmorphing, angles,quotes}

\newtheorem{theorem}{Theorem}[section]

\newtheorem{proposition}[theorem]{Proposition}
\newtheorem{lemma}[theorem]{Lemma}
\newtheorem{corollary}[theorem]{Corollary}

\theoremstyle{definition}
\newtheorem{remark}[theorem]{Remark}

\newtheorem{example}[theorem]{Example}

\def\F{\mathbb{F}}

\def\Q{\mathbb{Q}}

\def\Z{\mathbb{Z}}

\def\1{\mathbf{1}}

\def\<{\langle}
\def\>{\rangle}

\DeclareMathOperator{\mdeg}{mdeg}

\DeclareMathOperator{\SL}{SL}

\newcommand\nullset\varnothing

\begin{document}
	
	\title{Differential operators, anisotropy, and simplicial spheres}           
	
	\author{Kalle Karu, Matt Larson, and Alan Stapledon}

	\address{University of British Columbia}
	\email{karu@math.ubc.ca}
	
	\address{Princeton University and the Institute for Advanced Study}
	\email{mattlarson@princeton.edu}
	
	\address{Sydney Mathematics Research Institute}
	\email{astapldn@gmail.com}

	\begin{abstract}
	We find identities involving differential operators in the generic artinian reduction of the Stanley--Reisner ring of a simplicial sphere in any positive characteristic. These identities generalize the characteristic $2$ identities used by Papadakis and Petrotou to give a proof of the algebraic $g$-conjecture. We 
	show that 
	these identities are a shadow of an identity on the degree map, and we 
	use them 
	to prove the anisotropy of certain forms on the generic artinian reduction of the Stanley--Reisner ring and to prove weak Lefschetz results. 
	\end{abstract}
	
		\maketitle
	
	\vspace{-20 pt}
	
	\setcounter{tocdepth}{1}
%	\tableofcontents
	
	\section{Introduction}

In \cite{PapadakisPetrotoug}, Papadakis and Petrotou found some identities involving differential operators in the generic artinian reduction of the Stanley--Reisner ring of a simplicial sphere over a field of characteristic $2$, and they used those identities to give a proof of the algebraic $g$-conjecture for simplicial spheres. They conjectured a generalization of their identities, and this conjecture was proven in \cite{KaruXiaoAnisotropy}. These identities have enabled much recent progress in the study of simplicial spheres \cite{APP,LNS,Oba}.

We show that these identities are 
a shadow of 
 an identity on the degree map
that we prove in Theorem~\ref{thm:exact} below. 
Furthermore, we generalize this identity to generic artinian reductions of Stanley--Reisner rings of simplicial
spheres over a field of characteristic $p$ for any prime $p$.
We use these identities to deduce algebraic properties of such generic artinian reductions, including proving conjectures of Adiprasito, Papadakis, and Petrotou. 
Our results hold for all  connected oriented simplicial pseudomanifolds, and we work in that level of generality. 
Recall that a  simplicial pseudomanifold of dimension $d - 1$ is a pure simplicial complex $\Delta$ of dimension $d - 1$ 
such that every $(d - 2)$-dimensional face lies in exactly two facets, and each connected component of the geometric realization of $\Delta$ remains connected after we remove its $(d - 3)$-skeleton.

  Let $k$ be a field.
  Let $\Delta$ be a  connected  simplicial pseudomanifold of dimension $d - 1$ that is oriented over $k$  and has vertex set $V = \{1, \dotsc, n\}$. 
  Set $K = k(a_{i,j})_{1 \le i \le d, \, 1 \le j \le n}$. Let $K[\Delta]$ be the Stanley--Reisner ring of $\Delta$, and  set $\mu_i = \sum_{j} a_{i,j} x_j$ for $i \in \{1, \dotsc, d\}$, so $\mu_1, \dotsc, \mu_d$ is a linear system of parameters for $K[\Delta]$. Let $H(\Delta) \coloneqq K[\Delta]/(\mu_1, \dotsc, \mu_d)$ be the generic artinian reduction of $K[\Delta]$. Then $H^d(\Delta)$ is $1$-dimensional, and there is a degree map 
\[ \deg \colon K[x_1,\ldots, x_n]_d \to H^d(\Delta) \longrightarrow K\]
which induces a distinguished isomorphism from $H^d(\Delta)$ to $K$, that, with a slight abuse of notation, we also denote by $\deg$. Here $K[x_1,\ldots, x_n]_d$ denotes the degree $d$  part of the polynomial ring.  The degree map has the following concrete description: $H^d(\Delta)$ is generated as a $K$-vector space by the monomial $x^F$ associated to any facet $F = \{j_1 < \dotsb < j_d\}$ of $\Delta$. Fix an orientation of $\Delta$. This gives an ordering of the vertices of $F$, up to even permutation. Let $\epsilon_F \in \{ 1, -1 \}$ be the sign of the permutation that takes $(j_1,\ldots,j_d)$ to this ordering. Let $[F]$ denote the determinant of the $d \times d$ matrix whose $(i,m)$th entry is $a_{i,j_m}$. Then $\deg(x^F) = \epsilon_F/[F] \in K$.  

Let $L = (\nu_{i,j})$ be an $m \times n$ matrix of nonnegative integers for some positive integer $m$, i.e., the columns of $L$ are indexed by the vertices of $\Delta$. 
 Set 
\[ x^L \coloneqq \prod_{i,j} x_j^{\nu_{i,j}} = \prod_j x_j^{\sum_i \nu_{i,j}}.\] 
For example, if $F$ is a facet and $L$ is the $1 \times n$ matrix whose $j$th entry is $1$ if $j \in F$ and $0$ otherwise, then $x^L = x^F$.
We use $(x^L)^{1/p}$ to denote the monomial whose $p$th power is $x^L$ if such a monomial exists and zero otherwise. If $I = (I_{i,j})$ is a $d \times n$ matrix of nonnegative integers, then set
$$a^I \coloneqq \prod_{i,j} a_{i,j}^{I_{i,j}} \text{ and }I! \coloneqq \prod_{i,j} (I_{i,j})!. $$

\begin{theorem}\label{thm:exact} 
  Let $k$ be a field  of characteristic $p$.
  Let $\Delta$ be a  connected  simplicial pseudomanifold of dimension $d - 1$ that is oriented over $k$  and has vertex set $V = \{1, \dotsc, n\}$.
Let $J$ be a $1 \times n$ matrix of nonnegative integers with row sum $d$. Then
$$\deg(x^J) = (-1)^d \sum_I \deg((x^I \cdot x^J)^{1/p})^p \frac{a^I}{I!},$$
where $I$ varies over all $d \times n$ matrices of nonnegative integers whose row sums are all $p - 1$. 
\end{theorem}

In a different context, a similar identity appeared in \cite{APPS}. 
For each $i \in \{1, \dotsc, d\}$ and $j \in V$, the differential operator $\frac{\partial}{\partial a_{i,j}}$ acts on $K$. 
If $I = (I_{i,j})$ is a $d \times n$ matrix of nonnegative integers, define 
 \[ \partial^I \coloneqq \prod_{i,j} \left(\frac{\partial}{\partial a_{i,j}}\right)^{I_{i,j}}.\]
	In characteristic $p$, partial derivatives commute with $p$th powers. 
If all row sums of $I$ are $p-1$, then after applying $\partial^I$ to the identity in Theorem~\ref{thm:exact}, exactly one term on the right hand side survives, giving the following result.

 \begin{theorem}\label{thm:4.1} 
   Let $k$ be a field  of characteristic $p$.
   Let $\Delta$ be a  connected  simplicial pseudomanifold of dimension $d - 1$ that is oriented over  $k$   and has vertex set $V = \{1, \dotsc, n\}$.
 Let $I$ be a $d \times n$ matrix of nonnegative integers such that all row sums are $p - 1$. Let $J$ be a $1 \times n$  matrix of nonnegative integers  with row sum $d$. Then
 	\[
 	\partial^I \deg(x^J) = (-1)^d \deg( (x^I \cdot x^J)^{1/p} )^p.
 	\]	
 \end{theorem}

We will deduce Theorem~\ref{thm:exact} from Theorem~\ref{thm:4.1}, showing that, in fact, Theorem~\ref{thm:exact} and Theorem~\ref{thm:4.1} are equivalent. 
When $p=2$, the above theorem was conjectured in \cite[Conjecture~14.1]{PapadakisPetrotoug} and proved in \cite[Theorem~4.1]{KaruXiaoAnisotropy}, and it generalizes \cite[Proposition 5.1 and 5.7]{PapadakisPetrotoug}. 

\begin{example}
	Suppose that $d = 2$, $p = 3$, and $F = \{ 1 < 2 \}$ is a facet. Let $I = \begin{bmatrix}
		2 & 0 & 0 & \cdots \\
		0 & 2 & 0 & \cdots \\
	\end{bmatrix}$ and $J = (1,1,0,\cdots)$. Then Theorem~\ref{thm:4.1} says that $\partial^I \deg(x^F) = \deg(x^F)^3$. This corresponds to the fact that 
	$\left(\frac{\partial}{\partial a_{1,1}}\right)^2 \left(\frac{\partial}{\partial a_{2,2}}\right)^2 \frac{1}{a_{1,1}a_{2,2} - a_{2,1}a_{1,2}} = \frac{1}{(a_{1,1}a_{2,2} - a_{2,1}a_{1,2})^3}$ over a field of characteristic $3$.  
\end{example}

\smallskip
We use Theorem~\ref{thm:4.1} to deduce results about a quotient of $H(\Delta)$. 
Let $\overline{H}(\Delta)$ denote the \emph{Gorenstein quotient} of $H(\Delta)$, i.e., the quotient by the kernel of the pairing $H^{\bullet} \times H^{d - \bullet} \to K$ given by $(g, h) \mapsto \deg(g \cdot h)$. For example, if $\Delta$ is a homology sphere over $k$, then $\overline{H}(\Delta) = H(\Delta)$. 
Set $\ell = x_1 + \dotsb + x_n \in \overline{H}^1(\Delta)$. 

\begin{theorem}\label{thm:4.4}
  Let $k$ be a field  of characteristic $p$.
  Let $\Delta$ be a  connected  simplicial pseudomanifold of dimension $d - 1$ that is oriented over  $k$.
 Then for any nonzero $g \in \overline{H}^m(\Delta)$ with $m \le d/p$, 
$\ell^{d - pm} \cdot g^p$ is nonzero.
\end{theorem}
I.e., the form on $\overline{H}^m(\Delta)$ given by $g \mapsto \deg(\ell^{d - pm} \cdot g^p)$ is \emph{anisotropic}. When $p = 2$, this result recovers \cite[Theorem~4.4]{KaruXiaoAnisotropy}.

In particular, Theorem~\ref{thm:4.4} shows that if $g \in \overline{H}^m(\Delta)$ is nonzero for some $m \le d/p$, then $g^p$ is nonzero. This proves \cite[Conjecture 7.3]{APP}. 
Also, $\ell^{d - pm} \cdot g$ is nonzero, so 
the map from $\overline{H}^m(\Delta)$ to $\overline{H}^{d - (p-1)m}(\Delta)$ given by multiplication by $\ell^{d - pm}$ is injective. In particular, the map from $\overline{H}^m(\Delta)$ to $\overline{H}^{m + 1}(\Delta)$ given by multiplication by 
 $\ell$ is injective for $m \le \frac{d - 1}{p}$. This implies that 
the weak Lefschetz property holds for $\overline{H}(\Delta)$ in degrees at most  
$\lfloor (d-1)/p \rfloor$, 
and $\dim \overline{H}^0(\Delta) \le \dotsb \le \dim \overline{H}^{\lfloor (d-1)/p \rfloor}(\Delta)$. 
Even for $\mathbb{Q}$-homology spheres, these inequalities cannot in general be proven using the characteristic $2$ results of \cite{PapadakisPetrotoug} or \cite{KaruXiaoAnisotropy}, see Example~\ref{ex:RPd}. 
A proof of the strong Lefschetz property for $\overline{H}(\Delta)$ when $\Delta$ is a $\mathbb{Q}$-homology sphere was announced in \cite{APP}, see also \cite{Adiprasitog}. 

\smallskip

We can use Theorem~\ref{thm:4.1} to deduce results when $k$ has characteristic $0$ as well. For this, we need to know that the dimension of $\overline{H}(\Delta)$ does not depend on the characteristic of $k$.
 A result of Novik and Swartz \cite{NovikSwartzGorensteinRings}  guarantees that this holds when $\Delta$ 
 is a connected oriented simplicial $\Z$-homology manifold whose homology has no torsion. 
 Recall that  a $(d - 1)$-dimensional simplicial complex $\Delta$ is an $A$-homology manifold for some abelian group $A$ if the link of every nonempty face $G$ of $\Delta$ has the same homology as a sphere of dimension $d - |G| - 1$ over $A$.
  We can then prove the following result, which answers the question in \cite[Remark 14.3]{PapadakisPetrotoug} when $k$ has characteristic $0$.

\begin{theorem}\label{thm:tpower}
  Let $k$ be a field  of characteristic $0$.
  Let $\Delta$ be a  connected  simplicial $\Z$-homology manifold of dimension $d - 1$ that is oriented over $k$.
Assume that the integral homology of $\Delta$ is torsion-free.
For any $t \ge 2$ and nonzero $g \in \overline{H}^m(\Delta)$ with $m \le d/t$, $\ell^{d - tm} \cdot g^t$ 
is nonzero. 
\end{theorem}
This result, which is new even for simplicial spheres, is proved by using Theorem~\ref{thm:4.4} for each prime $p$ which divides $t$. 

\smallskip

Our paper is organized as follows. 
In Section~\ref{sec:identity}, we prove Theorem~\ref{thm:4.1} and then use it to deduce Theorem~\ref{thm:exact}.
In Section~\ref{sec:applications}, we prove Theorem~\ref{thm:4.4} and Theorem~\ref{thm:tpower}. See \cite[Section 2]{KaruXiaoAnisotropy} for background on pseudomanifolds, orientations, and degree maps. 

\subsection*{Acknowledgements} 
This work was conducted while the second and third authors were at the Institute for Advanced Study, where the third author received support from the Charles Simonyi endowment. The first author was supported by an NSERC Discovery grant. We thank the referees for their helpful comments.

\section{Proof of the identities}\label{sec:identity}

We begin by proving Theorem~\ref{thm:4.1}. 
Our argument imitates the proof of \cite[Theorem~4.1]{KaruXiaoAnisotropy}, but it requires some new ideas, see Remark~\ref{rem:compare}.

\subsection{Reductions}
Let $k$ be a field of characteristic $p$. 
We begin the proof of 
Theorem~\ref{thm:4.1}
by reducing to the case when $\Delta$ is the boundary of a simplex. 
It was shown in \cite[Theorem~3.2]{KaruXiaoAnisotropy} that the degree map for a connected  simplicial pseudomanifold $\Delta$ that is oriented over $k$ with $M$ facets can be expressed as a sum of degree maps
\begin{equation} \label{eq-conn-sum} \deg_\Delta = \deg_{\Pi_1} + \deg_{\Pi_2} + \dotsb + \deg_{\Pi_M}.\end{equation}
Here the summands correspond to facets $\sigma_i $ of $\Delta$. Let $0$ be a new vertex. Then the complex $\Pi_i$ corresponding to  $\sigma_i = \{ j_1, \ldots, j_d\}$ is the boundary of the $d$-simplex $\{ 0, j_1, \ldots, j_d\}$. The orientation on $\Pi_i$ is such that the orientations on $\sigma_i$ agree in $\Delta$ and $\Pi_i$. The map $\deg_{\Pi_i}$ is viewed as a function on $K[x_1,\ldots,x_n]_d$ via the composition
\[ K[x_1,\ldots,x_n]_d \to K[x_0, x_{j_1},\ldots, x_{j_d}] \stackrel{ \deg_{\Pi_i}}{\longrightarrow} K,\]
where the left map sets all $x_j$ that do not lie in $\{ x_{j_1}, \ldots, x_{j_d}\}$ equal to zero.

\begin{lemma}\label{lem:4.5} 
If the conclusion of Theorem~\ref{thm:4.1} holds for all $\Pi_i$, then it holds for $\Delta$.
\end{lemma}
\begin{proof}
After substituting in (\ref{eq-conn-sum}) and using linearity of $p$th powers, the statement of Theorem~\ref{thm:4.1}  is 
\[ \sum_{i=1}^{M} \partial^I \deg_{\Pi_i} (x^J) = \sum_{i=1}^{M} (-1)^d \deg_{\Pi_i}( (x^I \cdot x^J)^{1/p} )^p.\]
Using the conclusion of Theorem~\ref{thm:4.1} for each $\Pi_i$, we check that the corresponding summands on both sides are equal. If $x^J$ involves a variable $x_j$ such that $j$ is not a vertex of $\Pi_i$, then $\deg_{\Pi_i}$ vanishes on both sides. Similarly, if $x^I$ involves such a variable $x_j$, then the term on the right hand side vanishes; on the left hand side, $\deg_{\Pi_i}$ does not depend on variables $a_{i,j}$, so the derivative vanishes. 
In the case when $x^I x^J$ only involves variables $x_j$ 
such that $j$ is a vertex of $\Pi_i$, the two sides are equal by Theorem~\ref{thm:4.1} applied to $\Pi_i$.
\end{proof}

From now on, we will assume that $\Delta = \Pi$ is the boundary of a $d$-dimensional simplex with vertices $\{ 0, 1, \ldots, d \}$. Let  $A = (a_{i,j})_{1 \le i \le d, 0 \le j \le d}$, and note that we index the columns by $0,1,\ldots, d$. We use a similar convention 
for the matrices $I$ and $J$ in the statement of Theorem~\ref{thm:4.1}. 

The degree map for $\Delta=\Pi$ has a  simple form. For $0 \le j \le d$, let $X_j$ be $(-1)^j$ times the determinant of $A$ with the $j$th column removed. We write $X^J$ for a monomial in the $X_j$.  If $x^J$ has degree $d$, then 
\begin{equation} \label{eq-degPi}  
\deg_\Pi(x^J) = \frac{X^J}{X_0 X_1 \cdots X_d}.
\end{equation}
This formula, together with (\ref{eq-conn-sum}), can be used to give an explicit expression for the degree map $\deg_\Delta$. The resulting formula is the same as the one given by Lee \cite{Lee}. It can also be deduced from the integration map of Brion \cite{BrionStructurePolytopeAlgebra}. 

\begin{remark}\label{rem:specialize}
Note that the determinants $X_j$ lie in $\Z[a_{i,j}]$, and the expression for $\deg_\Delta$ is independent of the field $k$. This can be used to specialize the degree map from characteristic $0$ to characteristic $p$. For any $g\in \Z[a_{i,j}][x_j]$, we obtain the same result if we first reduce $g$ mod $p$ and then compute its degree in $\F_p(a_{i,j})$, or if we first compute its degree in $\Q(a_{i,j})$ and then reduce this mod $p$.
\end{remark}

The next theorem is a stronger version of Theorem~\ref{thm:4.1} in the case where $\Delta=\Pi$. When $p=2$, it was proved in \cite[Theorem~4.6]{KaruXiaoAnisotropy}. 
 
 \begin{theorem}\label{thm:4.6} 
 Let $I$ be a $d \times (d+1)$ matrix of nonnegative integers such that all row sums are $p - 1$. Let $J$ be a $1 \times (d+1)$  matrix of nonnegative integers  whose row sum is congruent to $-1  \pmod{p}$. Then
\[
\partial^I X^J = (-1)^d X^{pL}
\]
if there exists an integer vector $L$ such that $X^I X^J = X^{pL} (X_0 \cdots X_{d})^{p - 1}$, and $\partial^I X^J = 0$ otherwise.
\end{theorem}

\begin{lemma}\label{lem:4.7} 
	Theorem~\ref{thm:4.6} implies Theorem~\ref{thm:4.1} for $\Pi$.
\end{lemma}
\begin{proof}
The statement of Theorem~\ref{thm:4.1} is
	\[
	\partial^I  \frac{X^J}{X_0 \cdots X_{d}}  = 
	 \begin{cases}
		(-1)^d X^I X^J/(X_0 \cdots X_{d})^p &\textrm{ if } X^I X^J = X^{pL},\\
		0 &\textrm{ otherwise.}
	\end{cases}
	\]
	As partial derivatives commute with $p$th powers, multiplying both sides by $(X_0 \cdots X_{d})^p$ gives the equivalent statement
	\[
	\partial^I  \left( X^{J} (X_0 \cdots X_d)^{p-1} \right) = 
	\begin{cases}
		(-1)^d X^{pL}  &\textrm{ if } X^I X^J = X^{pL},\\
		0 &\textrm{ otherwise.}
	\end{cases}
	\]
	The above equation is the statement of Theorem~\ref{thm:4.6} for $I$ and $J'$ such that $X^{J'} = X^J (X_0 \cdots X_{d})^{p - 1}$. Note that the row sum of $J'$ is $d + (p - 1)(d + 1) = pd + (p - 1)$, which is congruent to $-1  \pmod{p}$. 
	\end{proof}

	\subsection{$\SL(d,k)$-invariance}
	Let $B\in \SL(d,k)$. Multiplication by $B$ from the left $A\mapsto B\cdot A$ gives a linear change of variables on the $a_{i,j}$, so it gives an action of $\SL(d,k)$ on the polynomial ring $k[a_{i,j}]$. If $k$ is infinite, then the invariants of this action are the subring $k[X_0, \dotsc, X_d]$  (\cite[Theorem 13.5.5]{Procesi2007}). 
Theorem~\ref{thm:4.6} depends only on the characteristic of the field $k$, not on the field itself. It suffices to prove it when $k$ is infinite.

	The following lemma generalizes \cite[Lemma~4.8]{KaruXiaoAnisotropy}. 
		
	\begin{lemma}\label{lem:4.8} 
	In the setup of Theorem~\ref{thm:4.6}, $\partial^I X^J \in k[a_{i,j}]$ is $\SL(d,k)$-invariant for any field $k$ of characteristic $p$. 
	\end{lemma}
	\begin{proof}
The group $\SL(d,k)$ is generated by elementary matrices (transvections) that act on $A$ by adding a constant $c$ times a row $r$ to a row $s$. It suffices to prove invariance under this action when $r=2$ and $s=1$.  Consider new variables
\[ a_{i,j}' = \begin{cases} 
a_{i,j} + c a_{2,j} & \text{if $i=1$,}\\
a_{i,j} & \text{otherwise}.
\end{cases} \]
Then the action maps the variable $a_{i,j}$ to $a'_{i,j}$. For a polynomial $f\in k[a_{i,j}]$, we denote by $f(a')$ the result of substituting $a'_{i,j}$ into $a_{i,j}$. We need to prove that 
\[ (\partial^I X^J)(a') = \partial^I X^J.\]
Let $f\in k[a_{i,j}]$. Then for any $0\leq m \leq d$, the chain rule gives
\begin{equation*}\begin{split}
 \frac{\partial}{\partial a_{2,m}} \big( f(a') \big) &= \frac{\partial}{\partial a_{2,m}} f(a_{1,j}+c a_{2,j}, a_{2,j}, \ldots) \\
& =  \left(\frac{\partial}{\partial a_{2,m}}  f \right) (a')  +  \left( c \frac{\partial}{\partial a_{1,m}}   f \right) (a')  =  \left( \left(\frac{\partial}{\partial a_{2,m}} + c \frac{\partial}{\partial a_{1,m}} \right)  f \right) (a') .
\end{split}\end{equation*}
For $i\neq 2$,
\[ \frac{\partial}{\partial a_{i,m}} \big( f(a') \big) = \left( \frac{\partial}{\partial a_{i,m}}  f \right) (a').\]
Now let $f = X^J$. We apply the derivatives in $\partial^I$ to $X^J(a')$ one at a time. The result is 
\[ \big( \partial^I(a') X^J \big) (a'),\]
where  $\partial^I(a')$ is the derivative $\partial^I$ with each  $\frac{\partial}{\partial a_{2,m}}$ replaced by $\frac{\partial}{\partial a_{2,m}} + c \frac{\partial}{\partial a_{1,m}}$, i.e.,
$$\partial^I(a')  \coloneqq \prod_{m = 0}^d \left(\frac{\partial}{\partial a_{2,m}} + c \frac{\partial}{\partial a_{1,m}} \right)^{I_{2,m}} 
\prod_{\substack{ i,m \\ i \not= 2}}
 \left(\frac{\partial}{\partial a_{i,m}}\right)^{I_{i,m}} . $$
The derivative $\partial^I(a')$ has one term equal to $\partial^I$. All other terms contain at least $p$ derivatives of the form $\frac{\partial}{\partial a_{1,m}}$ for some $m$. Lemma~\ref{lem:4.9.1} below shows that all these other terms vanish, so 
\[ \partial^I \big(X^J(a')\big) = (\partial^I X^J) (a').\]
Since $X^J$ is invariant under the $\SL(d,k)$ action, the left hand side is equal to $\partial^I X^J$.
	\end{proof}

	\begin{lemma}\label{lem:4.9.1} 
		If the row sum of $J$ is congruent to $-1  \pmod{p}$, then 
		$$\frac{\partial}{\partial a_{r,j_1}} \frac{\partial}{\partial a_{r,j_2}} \cdots \frac{\partial}{\partial a_{r,j_p}} X^J = 0. $$
 for any $1 \le r \le d$ and $0 \le j_1,\dotsc,j_p \le d$. 
	\end{lemma}
	\begin{proof}

Let $r$ be fixed. We argue by induction on the row sum of $J$. 

First assume that the row sum of $J$ is $p - 1$. 
Then since each $X_{j}$ is homogeneous of degree $1$ in the variables $\{ a_{r,j} : 0 \le j \le d \}$, $X^J$ is homogeneous of degree $p - 1$. Applying $\frac{\partial}{\partial a_{r,j}}$
to a homogeneous polynomial either kills it or decreases its degree by $1$, so 
$\frac{\partial}{\partial a_{r,j_1}} \frac{\partial}{\partial a_{r,j_2}} \cdots \frac{\partial}{\partial a_{r,j_p}} X^J = 0$. 

 Now assume that the row sum of $J$ is strictly greater than $p - 1$. We use relations between the $X_j$ to decrease the row sum.

Let $\tilde{A}$ be the matrix $A = (a_{i,j})$  augmented with an extra column ${e}_r$ (the standard basis vector). Denote by $Y_{j m}$ the determinant of $\tilde{A}$ with two columns $j$ and $m$ removed.  The determinants $Y_{j m}$ satisfy the classical Pl\"ucker relations: for any $0\leq j_1< j_2 < j_3 < j_4 \leq d+1$, we have
 \[ Y_{j_1 j_2} Y_{j_3 j_4} - Y_{j_1 j_3} Y_{j_2 j_4} +Y_{j_1 j_4} Y_{j_2 j_3} =0.\]
  We specialize  to the case  where $(j_1, j_2, j_3, j_4) = (0,1,j,d+1)$ for some $1 <  j < d+1$. Using that  $X_j = (-1)^j Y_{j,d+1}$, the relation becomes
  \[ (-1)^j Y_{01} X_{j} + Y_{0 j} X_{1}  + X_{0} Y_{1 j} =0.\]
This relation allows us to write
\begin{equation}\label{eq:X_kX_1X_2}
	X_j = b_{1} X_0 + b_{2} X_1,
\end{equation}
where $b_{1} = (-1)^{j+1} Y_{1j}/Y_{01}$ and $b_{2} = (-1)^{j+1}  Y_{0j}/Y_{01}$ are elements of $K$ that are independent of the variables $\{ a_{r,j} : 0 \le j \le d \}$. 
Substituting the expression from \eqref{eq:X_kX_1X_2} into $X^J$ for all $j > 1$, we reduce to the case when $X^J = X_0^{m_1} X_1^{m_2}$ for some nonnegative integers $m_1,m_2$ such that $m_1 + m_2$ is the row sum of $J$. Since the row sum of $J$ is $-1  \pmod{p}$ and is strictly greater than $p - 1$, it must be at least $2p - 1$. 
 Hence either $m_1 \geq p$ or $m_2 \geq p$. Without loss of generality, assume that $m_1 \geq  p$. Then we can write $X^J = X_0^p X^{J'}$ with the row sum of $J'$ congruent to  $-1  \pmod{p}$ and strictly less than the row sum of $J$. By induction, and since derivatives commute with $p$th powers, we have
 \[
 \frac{\partial}{\partial a_{r,j_1}} \frac{\partial}{\partial a_{r,j_2}} \cdots \frac{\partial}{\partial a_{r,j_p}} X^J = X_0^p  \frac{\partial}{\partial a_{r,j_1}} \frac{\partial}{\partial a_{r,j_2}} \cdots \frac{\partial}{\partial a_{r,j_p}} X^{J'} = 0. \qedhere
 \]			
	\end{proof}

\subsection{Proof of Theorem~\ref{thm:4.6}}
We will now prove Theorem~\ref{thm:4.6}, which also finishes the proof of Theorem~\ref{thm:4.1}.

Recall the setup. Let $I = (I_{i,j})$ be a $d \times (d + 1)$ matrix of nonnegative integers such that all row sums are $p - 1$. Let $J$ be a $1 \times (d + 1)$  matrix of nonnegative integers  whose row sum is $-1  \pmod{p}$. 
We write
\[
X^I X^J = X^M (X_0 \cdots X_d)^{p - 1},
\]
for some $1 \times (d+1)$ matrix of integers $M$. Observe that if $(X^M)^{1/p}$ exists, then, since $X^M (X_0 \cdots X_d)^{p - 1}$ has nonnegative exponents, the coefficients of $M$ are nonnegative integers.
Theorem~\ref{thm:4.6} states that $\partial^I X^J$ is equal to $(-1)^d X^M$ if  the $p$th root $(X^M)^{1/p}$ exists, and it is zero otherwise. 

We give $k[a_{i,j}]$ a multigrading by $\Z^{d+1}$, where $a_{i,j}$ has degree $e_j$ (the standard basis vector). Then $X_j$ lies in degree $(1,1,\ldots,1)-e_j$. Since the degrees of the $X_j$ for $j=0,\ldots, d$ are linearly independent, there is at most one monomial in $k[X_0,\ldots, X_d]$ in each degree.

Lemma~\ref{lem:4.8} implies that $\partial^I X^J$ lies in the ring $k[X_0,\ldots, X_d]$. Since $X^J$ is homogeneous with respect to the multigrading and each partial derivative with respect to $a_{i,j}$ reduces the degree by $e_j$, it follows that $\partial^I X^J$ is also homogeneous. Its multidegree is
\begin{gather*} 
\mdeg \partial^I X^J = \mdeg X^J - \sum_{i,j} I_{i,j} e_j = \mdeg X^J + \mdeg X^I - d(p-1)(1,1,\ldots,1) \\
= \mdeg X^I X^J - \mdeg (X_0\cdots X_d)^{p-1} = \mdeg X^M.
\end{gather*}
It follows that $\partial^I X^J$ is a constant multiple of the monomial $X^M$. We need to prove that this constant is equal to $(-1)^d$ if $(X^M)^{1/p}$ exists, and it is zero otherwise.

\begin{lemma}\label{lem:4.10} 
If $(X^M)^{1/p}$ does not exist, then $\partial^I X^J = 0$.
\end{lemma}
\begin{proof}
Lemma~\ref{lem:4.9.1}  gives that the partial derivative of $\partial^I X^J$ with respect to any $a_{i,j}$ is zero. This implies that $\partial^I X^J$ is a $p$th power in $k[a_{i,j}]$ as $\partial^I X^J$ has coefficients in the perfect field $\mathbb{F}_p$. Since the $X_j$ are irreducible, it is also a $p$th power in $k[X_0,\ldots, X_d]$. As $\partial^I X^J$ is a constant multiple of $X^M$, we see that the constant must be $0$ in order for it to be a $p$th power. 
\end{proof}

\begin{lemma}\label{lem:4.11}
	If $X^M = X^{pL}$ for some $L$, then $\partial^I X^J = (-1)^d X^{pL}$. 
\end{lemma}

We have seen that $\partial^I X^J$ is a constant multiple of the monomial $X^M = X^{pL}$.
We consider $I, J, L$ satisfying the equation
\begin{equation}\label{eq-deg} 
X^I X^J = X^{pL} (X_0 \cdots X_d)^{p - 1}.
\end{equation}
For every $I$, there is a unique monomial $X^J$ of smallest degree such that this equation has a solution for some $L$. Any other $X^J$ is the product of this minimal degree monomial with the $p$th power of a monomial.
This implies that if $I, J, L$ satisfy (\ref{eq-deg}), then
\[ \partial^I X^J = c_I X^{pL} \]
for some constant $c_I$ that only depends on $I$. To prove that $c_I = (-1)^d$ as stated in the lemma, we compare the constants $c_I$ for different $I$. 

Recall that $X_j$ is  $(-1)^j$ times the determinant of the matrix $A$ with column $j$ removed. 
We write $\overline{X}_j$, $j=0,\ldots, d-1,$ for similar signed determinants in dimension $d-1$, i.e., $(-1)^j$ times the determinant of the result of removing the $j$th column, the last column, and the $r$th row from $A$. Assume that the matrix $I$ has last column equal to zero, and let $\overline{I}_{\hat{r}}$ be the matrix obtained from $I$ by removing the last column and row $r$ for some $1\leq r \leq d$. Then, by the discussion above with $d$ replaced by $d-1$, there is a unique monomial $\overline{X}^{\tilde{J}}$ of smallest degree so that $\overline{X}^{\overline{I}_{\hat{r}}} \overline{X}^{\tilde{J}} = \overline{X}^{p\overline{L}} (\overline{X}_0 \cdots \overline{X}_{d-1})^{p - 1}$, and we have 
\[ \partial^{\overline{I}_{\hat{r}}} \overline{X}^{\tilde{J}} = c_{\overline{I}_{\hat{r}}} \overline{X}^{p\overline{L}} \]
for some $\overline{L}$ and some constant $c_{\overline{I}_{\hat{r}}}$.

\begin{lemma} \label{lem-coeff}
In the situation above, we have 
\[ c_{\overline{I}_{\hat{r}}} = - c_I.\]
\end{lemma}

Let us prove Lemma~\ref{lem:4.11} using Lemma~\ref{lem-coeff}.

\begin{proof}[Proof of Lemma~\ref{lem:4.11}]
We prove the lemma in dimension $d-1$: we will prove it for an arbitrary matrix $\overline{I}$ of size $(d-1)\times d$.
Create a matrix $I$ of size $d\times(d+1)$ by adding to $\overline{I}$ an extra row with row sum $p - 1$ and a zero column. From the matrix $I$ we can construct the matrix $\overline{I}_{\hat{r}}$ by removing the last column and row $r$. One choice of $r$ takes us back to the original $\overline{I}$. Other choices of $r$ correspond to removing from $\overline{I}$ a row and adding another row. Lemma~\ref{lem-coeff} states that the constants $c_{\overline{I}_{\hat{r}}}$ are the same for all choices of $r$. We can change rows of $\overline{I}$ one by one, without changing the constant $c_{\overline{I}}$, until the matrix $\overline{I}$ is given by 
\[ \overline{I}_{i,j} = \begin{cases} p - 1 & \text{ if $i=j$,}\\ 0 & \text{ otherwise.} \end{cases}\]
Recall that we index columns of $\overline{I}$ by $0,1,\ldots, d-1$, so column $0$ of $\overline{I}$ is zero. For this $\overline{I}$ we may take $X^J = X_0^{p-1}$ and $X^L = 1$. Then
\[ \partial^{\overline{I}} X^J = \big( (p-1)! \big)^{d-1} X^{pL}.\]
As $(p-1)!\equiv -1 \pmod{p}$, we get that $c_{\overline{I}}=(-1)^{d-1}$.
\end{proof}

\begin{proof}[Proof of Lemma~\ref{lem-coeff}]
Let $\overline{J}$ and $\overline{L}$ be the vectors $J$ and $L$ with their last entries removed. Let $I_i$ be the $i$th row of $I$, and let $\overline{I}_i$ be the row $I_i$ with its last entry removed. 
We have $\overline{X}^{\overline{I}_{\hat{r}}} \overline{X}^{\overline{J} + \overline{I}_r} = \overline{X}^{p\overline{L}} (\overline{X}_0 \cdots \overline{X}_{d-1})^{p - 1}$, and $\overline{X}^{\overline{J} + \overline{I}_r}$ has the smallest degree with this property. The vectors $\overline{I}_i$ for $i\neq r$ form the rows of $\overline{I}_{\hat{r}}$.

Since the last column of $I$ is zero, we may assume that the last entry of $J$ is $p-1$ and the last entry of $L$ is zero. By definition, 
\begin{equation}\label{eq-ci} \partial^I  X^J = c_I X^{pL}. \end{equation}
Consider the $k$-algebra map $\pi \colon k[a_{i,j}]_{1 \le i \le d, \, 0 \le j \le d} \to k[a_{i,j}]_{1 \le i \le d, \, 0 \le j \le d - 1}$ defined by
\[
\pi(a_{i,j}) = \begin{cases}
	a_{i,j} &\textrm{ if } j < d, \\
	(-1)^{r + d} &\textrm{ if } (i,j) = (r,d), \\
	0 &\textrm{ otherwise.}
\end{cases}
\]
We claim that if we apply $\pi$ to both sides of equation \eqref{eq-ci}, we get
\begin{equation}\label{eq-cirv2} \partial^{\overline{I}_{\hat{r}}}  \overline{X}^{\overline{J}+ \overline{I}_r} = -c_I \overline{X}^{p\overline{L}}. \end{equation}
Here $\overline{I}_{\hat{r}}, \overline{J}+ \overline{I}_r, \overline{L}$ satisfy equation (\ref{eq-deg}) in dimension $d-1$, hence $c_{\overline{I}_{\hat{r}}} = - c_I$.

It remains to establish \eqref{eq-cirv2}. 
By the cofactor expansion of the determinant, we compute $\pi(X_j) = \overline{X}_j$ for $0 \le j \le d - 1$, so applying $\pi$ to the 
right hand side of \eqref{eq-ci} gives $c_I \overline{X}^{p\overline{L}}$. 

 Since $\partial^I$ does not involve the variables $\{ a_{i,d} : 1 \le i \le d \}$, applying $\partial^I$ commutes with applying $\pi$. Write $\partial^{I_r}$  for the derivative corresponding to row $r$ of $I$.  Applying $\pi$ to the left hand side of \eqref{eq-ci} gives
\[  \partial^I \pi(X^J) = 	\partial^{\overline{I}_{\hat{r}}} \left(\overline{X}^{\overline{J}} \left(\partial^{I_r} \pi(X_d)^{p - 1} \right) \right). \\ 
 \]
  We use that $(p-1)! \equiv -1 \pmod{p}$ and $(-1)^{p-1} \equiv 1 \pmod {p}$ to compute
 \begin{align*}
 	\partial^{I_r} \pi(X_d^{p - 1}) = (-1)^{(p - 1)(d + r + 1)}\partial^{I_r} \left(\sum_{j = 0}^{d - 1} a_{r,j} \overline{X}_j \right)^{p - 1} 
 	= (p - 1)! \overline{X}^{\overline{I}_r} = - \overline{X}^{\overline{I}_r}. 
 \end{align*} 
We conclude that   $\partial^I \pi(X^J) = 	-\partial^{\overline{I}_{\hat{r}}} \overline{X}^{\overline{J} + \overline{I}_r}$, which establishes \eqref{eq-cirv2}. 
\end{proof}

\begin{remark}\label{rem:compare}
The proof of Theorem~\ref{thm:4.1} is similar to the argument in \cite[Section 4]{KaruXiaoAnisotropy}, but there are a few places where new arguments are needed, especially Lemma~\ref{lem:4.9.1} and Lemma~\ref{lem:4.11}. For example, the proof of \cite[Lemma 4.11]{KaruXiaoAnisotropy}, which is the special case of Lemma~\ref{lem:4.11} when $p=2$, uses that $\mathbb{F}_2$ has size $2$ and so all nonzero elements of $\mathbb{F}_2$ are equal to $1$. 
\end{remark}

\subsection{Proof of Theorem~\ref{thm:exact}}
We now deduce Theorem~\ref{thm:exact} from Theorem~\ref{thm:4.1}. For any $r \ge 1$ such that $r!$ is invertible in $k$, consider the differential operator
$$\Psi_r \coloneqq \sum_I \frac{a^I}{I!} \partial^I, $$
where $I$ varies over all $d \times n$ matrices of nonnegative integers such that all row sums are equal to $r$. We prove a more general identity involving $\Psi_r$. 
\begin{theorem}\label{thm:psiidentity}
Let $k$ be a field in which $r!$ is invertible.
   Let $\Delta$ be a  connected  simplicial pseudomanifold of dimension $d - 1$ that is oriented over $k$  and has vertex set $V = \{1, \dotsc, n\}$.
Let $J$ be a $1 \times n$ matrix of nonnegative integers with row sum $d$. Then
$$\Psi_r \deg(x^J) = (-1)^{dr}  \deg(x^J).$$
\end{theorem}

\begin{proof}[Proof of Theorem~\ref{thm:exact}]
Applying $\Psi_{p-1}$ to $\deg(x^J)$, we have
\begin{equation*}\begin{split}
\Psi_{p-1} \deg(x^J) &= \sum_{I} \frac{a^I}{I!} \partial^I \deg(x^J) \\ 
&= (-1)^d \sum_I \deg((x^I \cdot x^J)^{1/p})^p \frac{a^I}{I!},
\end{split}\end{equation*}
where $I$ varies over all $d \times n$ matrices of nonnegative integers such that all row sums are equal to $p-1$, and the second equality is by Theorem~\ref{thm:4.1}. As $\Psi_{p-1} \deg(x^J) = (-1)^{d(p-1)} \deg(x^J)$ by Theorem~\ref{thm:psiidentity} and $(-1)^{d(p-1)} \equiv 1 \pmod p$, the result follows. 
\end{proof}
In order to prove Theorem~\ref{thm:psiidentity}, we introduce some auxiliary differential operators that will allow us to factor $\Psi_r$. For $1 \le i \le d$, define 
$$\Phi_{i,r} \coloneqq \sum_{1 \le j_1, \dotsc, j_r \le n} a_{i,j_1} \dotsb a_{i,j_r} \frac{\partial}{\partial a_{i,j_1}} \dotsb \frac{\partial}{\partial{a_{i,j_r}}}.$$
Note that if $i \not= i'$, then $\Phi_{i,r}$ and $\Phi_{i', r}$ commute, as they involve distinct variables. We have that
\begin{equation}\label{eq:psifactor}
(r!)^d \Psi_r = \prod_{i=1}^{d} \Phi_{i,r}.	
\end{equation}
For $1 \le i \le d$, define
$\phi_i \coloneqq \sum_{j = 1}^n a_{i,j} \frac{\partial}{\partial {a_{i,j}}}.$  Then we have the following identity.
\begin{lemma}\label{lem:phifactor}
For any $1 \le i \le d$ and any $r \ge 1$, we have
$$\Phi_{i, r} = \phi_i (\phi_i - 1) \dotsb (\phi_i - (r-1)).$$
\end{lemma}

\begin{proof}
To simplify notation, let $\partial_j= a_{i,j} \frac{\partial}{\partial a_{i,j}}$. Note that the $\partial_j$ commute with each other. The one-variable case, i.e., that
\[ a_{i,j}^\alpha \left (\frac{\partial}{\partial a_{i,j}}\right)^\alpha = \partial_j(\partial_j-1)\cdots(\partial_j-(\alpha-1))\]
for an integer $\alpha\geq 0$, can be checked by applying both sides to monomials $a_{i,j}^\beta$. 

We will further use the notation 
\[ [x]_\alpha = x(x-1) \cdots (x-(\alpha-1)),\quad  {x \choose \alpha} = \frac{[x]_\alpha}{\alpha!},\]
where $\alpha\geq 0$ is an integer and $x$ is a variable.

The sum for $\Phi_{i,r}$ can be expressed as
\[ \Phi_{i,r} = \sum_{\alpha_1+\cdots + \alpha_n=r} {r \choose \alpha_1, \ldots, \alpha_n} a_{i,1}^{\alpha_1} \left (\frac{\partial}{\partial a_{i,1}}\right )^{\alpha_1} \cdots a_{i,n}^{\alpha_n} \left (\frac{\partial}{\partial a_{i,n}}\right )^{\alpha_n}.\]
Here the integers $\alpha_j$ are nonnegative, and the multinomial appears in the sum because we need to choose $\alpha_1$ indices $j$ equal to $1$, $\alpha_2$ indices $j$ equal to $2$, and so on. Applying the one variable formula above, this simplifies to 
\begin{align*} 
\Phi_{i,r} &= \sum_{\alpha_1+\cdots + \alpha_n=r} {r \choose \alpha_1, \ldots, \alpha_n} [\partial_1]_{\alpha_1} \cdots [\partial_n]_{\alpha_n} \\ 
&= \sum_{\alpha_1+\cdots + \alpha_n=r} r! {\partial_1 \choose \alpha_1} \cdots {\partial_n \choose \alpha_n}.
\end{align*}
The claim is that this sum is equal to 
\[ [\sum_j \partial_j]_r,\]
or after dividing by $r!$,
\[ {\sum_j \partial_j \choose r} =  \sum_{\alpha_1+\cdots + \alpha_n=r} {\partial_1 \choose \alpha_1} \cdots {\partial_n \choose \alpha_n}. \]
This last equality clearly holds if we substitute any nonnegative integers into $\partial_j$. 

Because the $\partial_j$ commute, both sides of the equality in the statement of the lemma can be expressed as polynomials in $\partial_j$ with integer coefficients. Since these polynomials take the same value when evaluated at the Zariski dense set of all nonnegative integers, they must be equal. Hence they also reduce to equal polynomials over any field.
\end{proof}
\begin{proof}[Proof of Theorem~\ref{thm:psiidentity}]
As in the proof of Theorem~\ref{thm:4.1}, we reduce to the case of the boundary of a simplex. We can write the degree map on $\Delta$ as a sum of the degree maps on boundaries of $d$-simplices $\Pi_i$, see \eqref{eq-conn-sum}. Let $\tilde{\Psi}_r$ be the differential operator on $K[x_0, \dotsc, x_n]$ given by $\sum_I \frac{a^I}{I!} \partial^I$, where the sum is over all $d \times (n+1)$ matrices $I$ with all row sums equal to $r$, i.e., $\tilde{\Psi}_r$ is given by the same formula as $\Psi_r$ except we include $\frac{\partial}{\partial {a_{i,0}}}$ as well. For each $i$, $\tilde{\Psi}_r \deg_{\Pi_i}(x^J)$ is equal to the result of applying the differential operator which is defined only using the vertices of $\Pi_i$. This holds for $\Delta$ as well. We see that
$$\Psi_r \deg_{\Delta}(x^J) = \tilde{\Psi}_r \deg_{\Delta}(x^J) = \sum_{i=1}^{M} \tilde{\Psi}_r \deg_{\Pi_i}(x^J).$$
If we know that $\tilde{\Psi}_r \deg_{\Pi_i}(x^J) = (-1)^{dr} \deg_{\Pi_i}(x^J)$, then we deduce $\Psi_r \deg_{\Delta}(x^J) = (-1)^{dr}  \deg_{\Delta}(x^J)$ as well. 
We can therefore reduce to showing that $\Psi_r \deg(x^J) = (-1)^{dr}  \deg(x^J)$ when $\Delta = \Pi$ is the boundary of a simplex with vertex set $\{0, 1, \dotsc, d\}$.
We claim that, for each $1 \le i \le d$, we have
$$\phi_i \deg(x^J) = - \deg(x^J).$$
Indeed, cofactor expansion along the $i$th row implies that $\phi_i X_j = X_j$ for any $j$. The product rule then implies that $\phi_i$ acts on any Laurent monomial of degree $m$ in the $X_j$ by multiplication by $m$. 
By \eqref{eq-degPi}, $\deg(x^J)$ is a Laurent monomial of degree $-1$ in the $X_j$. By Lemma~\ref{lem:phifactor}, we have
\begin{equation*}\begin{split}
\Phi_{i,r} \deg(x^J) &= \phi_i(\phi_i - 1) \dotsb (\phi_i - (r - 1)) \deg(x^J) \\  
&= (-1)(-2) \dotsb (-r) \deg(x^J) \\  
&= (-1)^{r} r! \deg(x^J).
\end{split}\end{equation*}
Applying \eqref{eq:psifactor} then gives the result.  
\end{proof}

\section{Applications}\label{sec:applications}

\subsection{Applications in characteristic $p$}

Fix a prime $p$, and assume that $k$ is a field of characteristic $p$. Let $\Delta$ be a connected simplicial pseudomanifold of dimension $d-1$ that is oriented over $k$  with vertex set $V = \{1, \dotsc, n\}$.
We apply Theorem~\ref{thm:4.1} to deduce results about the Gorenstein quotient $\overline{H}(\Delta)$. Firstly, we have the following generalization of \cite[Corollary~4.2]{KaruXiaoAnisotropy}. 

 \begin{corollary}\label{cor:4.2} 
 	Let $h \in \overline{H}^m(\Delta)$ for some $m \le d/p$. 
 Let $J$ be a $1 \times n$  matrix of nonnegative integers  such that the row sum is $d - pm$. 
  Then
 $$\deg(h^p \cdot x^J) = (-1)^d \sum_I \deg(h \cdot  (x^I \cdot x^J)^{1/p})^p \frac{a^I}{I!},$$
 where $I$ varies over all $d \times n$ matrices of nonnegative integers whose row sums are all $p - 1$. 
 Moreover, if $I$ is  a $d \times n$ matrix of nonnegative integers such that all row sums are $p - 1$, then 
  \[
 \partial^I \deg(h^p \cdot x^J) = (-1)^d \deg( h \cdot (x^I \cdot x^J)^{1/p} )^p. 
 \]
 
 \end{corollary}
 \begin{proof}
 	The second identity follows by applying  $\partial^I$ to the first identity, and using the fact that partial derivatives commute with $p$th powers in characteristic $p$. 

 	We prove the first identity. 	Write $h = \sum_L \lambda_L x^L$ for some $1 \times n$ matrices of nonnegative integers $L$ and some $\lambda_L \in K$. Using Theorem~\ref{thm:exact}, we compute
 	\begin{align*}
 		\deg(h^p \cdot x^J) &= \sum_L  \lambda_L^p  \deg(x^{pL} \cdot x^J) \\
 		&= (-1)^d \sum_L  \lambda_L^p \sum_I \deg(x^L \cdot  (x^I \cdot x^J)^{1/p})^p \frac{a^I}{I!} \\
 		&= (-1)^d \sum_I \deg(h \cdot  (x^I \cdot x^J)^{1/p})^p \frac{a^I}{I!}. \qedhere
 	\end{align*}
 	 \end{proof}

Secondly, we have the following generalization of \cite[Corollary~4.3]{KaruXiaoAnisotropy}. 

\begin{corollary}\label{cor:4.3}
Suppose $g \in \overline{H}^m(\Delta)$ has $\ell^r \cdot g^p = 0$ in $\overline{H}^{r + pm}(\Delta)$ for some $0 \le r \le d - pm$. 
Let $t = \lfloor r/p \rfloor$. Then $\ell^t \cdot  g = 0$ in $\overline{H}^{t + m}(\Delta)$.
\end{corollary}
\begin{proof}
Write $r = tp + s$, where $t = \lfloor r/p \rfloor$  and $0 \leq s < p$. 
For any  $d \times n$ matrix $I$ of nonnegative integers with all row sums equal to $p - 1$ and 
any $1 \times n$ matrix $J$ of nonnegative integers with row sum $d - pm - r$, we have
\[ 
0 = \deg(\ell^r \cdot g^p \cdot  x^J) = \deg( (\ell^t \cdot g)^p \cdot x^J \cdot  \ell^s )
= \sum_{(v_1,\ldots,v_s)} \deg( (\ell^t \cdot g)^p \cdot x^J  \cdot x_{v_1} \cdots  x_{v_s}),
\]
where $(v_1,\ldots,v_s)$ varies over all sequences of $s$ vertices in $V$. 
By Corollary~\ref{cor:4.2}, we have
\begin{equation}\label{eq:g^p}
 0 = \sum_{(v_1,\ldots,v_s)} \partial^I \deg( (\ell^t \cdot g)^p \cdot x^J \cdot x_{v_1} \cdots  x_{v_s}) 
= \sum_{(v_1,\ldots,v_s)} \deg( \ell^t \cdot g \cdot ( x^I \cdot x^J \cdot  x_{v_1} \cdots  x_{v_s} )^{1/p} )^p.
\end{equation}

Suppose that $\ell^t \cdot g$ is nonzero.
By the nondegeneracy of the pairing between $\overline{H}^{m + t}(\Delta)$ and  $\overline{H}^{d - m - t}(\Delta)$, there is a $1 \times n$ matrix $L$ of nonnegative integers with row sum $d - m - t$ such that $\deg(\ell^t \cdot g \cdot x^L)$ is nonzero. 
Let $L = (\nu_1,\ldots, \nu_n)$, so $x^L = x_1^{\nu_1} \cdots  x_n^{\nu_n}$. Let $M = (m_1,\ldots, m_n)$ be a sequence of nonnegative integers that sums to $s$ with the property that $x^M$ divides $x^{pL}$. For example, if $\nu_i$ is nonzero, then one can take $M = (0, \ldots,0,s,0,\ldots ,0)$, where the $s$ is in the $i$th coordinate. 
Choose $I$ and $J$ so that
\begin{equation}\label{eq:LMIJ}
x^{pL}/x^M = x^I \cdot x^J 
\end{equation} 
Note that both sides have degree $p(d - m) - r$, so this is always possible. 

Now,  $x^I \cdot x^J \cdot   x_{v_1} \cdots  x_{v_s} = (x^{pL}/x^M) \cdot x_{v_1} \cdots  x_{v_s}$ is a $p$th power if and only if $x^{(p - 1)M} \cdot x_{v_1} \cdots  x_{v_s}$ is a $p$th power. The latter holds if and only if $x_{v_1} \cdots  x_{v_s} = x^M$. Indeed, if $x_{v_1} \cdots  x_{v_s} = x_1^{b_1} \cdots  x_n^{b_n}$ for some nonnegative integers satisfying $b_1 + \cdots + b_n = s$, then 
$x^{(p - 1)M} \cdot x_{v_1} \cdots  x_{v_s}$ is a $p$th power if and only if $b_i \equiv m_i \pmod p$ for all $i$. Since $b_i \leq s < p$, this holds if and only if $b_i = m_i$ for all $i$. 
 
The number of sequences $(v_1,\ldots,v_s)$ 
such that $x_{v_1} \cdots  x_{v_s}$ is equal to $x^M$ is the multinomial coefficient $s!/((m_1)! \cdots  (m_n)!)$, which is nonzero in $k$. 
Therefore \eqref{eq:g^p} implies that
\[
 0 = \frac{s!}{(m_1)!\cdots  (m_n)!} \deg( \ell^t \cdot g \cdot ( x^I\cdot  x^J \cdot  x^M)^{1/p} )^p = \frac{s!}{(m_1)!\cdots  (m_n)!} \deg(\ell^t  \cdot g \cdot x^L)^p.
 \]
This contradicts the assumption that $\deg(\ell^t \cdot g \cdot x^L)$ is nonzero.
\end{proof}

 \begin{proof}[Proof of Theorem~\ref{thm:4.4}]
 	Suppose that $g \in \overline{H}^m(\Delta)$ is nonzero. We claim that $\ell^r \cdot g^p  \in \overline{H}^{r + pm}(\Delta)$ is nonzero for any $r \le d - pm$. The case when $r = 0$ follows from Corollary~\ref{cor:4.3}. Assume that $r > 0$ and let $t = \lfloor r/p \rfloor < r$. By induction, $\ell^t \cdot  g^p = (\ell^t \cdot g) \cdot g^{p - 1}$ is nonzero, so Corollary~\ref{cor:4.3} implies that $\ell^r \cdot g^p$ is nonzero.
 \end{proof}

\subsection{Applications in characteristic $0$}

We now give applications of our results to simplicial $\Z$-homology  manifolds in characteristic $0$. Our strategy is similar to \cite[Section 5]{KaruXiaoAnisotropy}. 
The following application of results of Novik and Swartz allows us to transfer information from characteristic $p$ to characteristic $0$.   

\begin{proposition}\cite[Theorem 1.3 and 1.4]{NovikSwartzGorensteinRings} \label{prop:NS}
Let $k$ be  a field.
   Let $\Delta$ be a  connected  simplicial $k$-homology manifold of dimension $d - 1$ that is oriented over  $k$   and has vertex set $V = \{1, \dotsc, n\}$.
    Let $K = k(a_{i,j})_{1 \le i \le d, \, 1 \le j \le n}$. 
Let $\mu_1, \dotsc, \mu_d$ be an l.s.o.p. for $K[\Delta]$. Then the dimension of the Gorenstein quotient of $K[\Delta]/(\mu_1, \dotsc, \mu_d)$ depends only on the number of faces of each dimension of $\Delta$ and the Betti numbers of $\Delta$ over $k$.
\end{proposition}

If $\Delta$ is a simplicial 
complex whose integral homology has no $p$-torsion, then the universal coefficient theorem implies that its Betti numbers are the same in characteristic $0$ and characteristic $p$. 
An orientation of $\Delta$ in characteristic $0$ also gives an orientation of $\Delta$ in characteristic $p$. Proposition~\ref{prop:NS} then gives the following corollary. 

\begin{corollary}\label{cor:dimensioninvariant}
Let $\Delta$ be a connected oriented simplicial $\Z$-homology manifold whose homology has no $p$-torsion. Then the dimensions of the graded pieces of the Gorenstein quotient of a generic artinian reduction of $K[\Delta]$ are the same when $k$ has characteristic $0$ or characteristic $p$. 
\end{corollary}

\begin{proposition}\label{prop:pansitropychar0}
Let $\Delta$ be a connected oriented simplicial $\Z$-homology manifold of dimension $d-1$ whose homology has no $p$-torsion. Let $k$ be a field of characteristic $0$. For any nonzero $g \in \overline{H}^m(\Delta)$ with $m \le d/p$, $\deg(\ell^{d - pm} \cdot g^p)$ is nonzero. 
\end{proposition}

\begin{proof}
Because $k$ has characteristic $0$, it is an extension of $\mathbb{Q}$. By the Chevalley extension theorem (see, e.g., \cite[Theorem 3.1.1]{ValuedFields}), we can extend the $p$-adic valuation on $\mathbb{Q}$ to a valuation on $k$. Let $\kappa$ denote the residue field. 

Let $\overline{H}(\Delta)$ be the Gorenstein quotient of the generic artinian reduction of ${k}(a_{i,j})[\Delta]$, and define $\overline{H}_{\kappa}(\Delta)$ similarly.
Choose a basis for $\overline{H}^m_{\kappa}(\Delta)$ and for $\overline{H}^{d - m}_{\kappa}(\Delta)$ consisting of monomials. Denote these bases by $v_1, \dotsc, v_r$ and $w_1, \dotsc, w_r$, where $r = \dim \overline{H}^m_{\kappa}(\Delta)$. Let $M_{\kappa}$ be the matrix whose $(i, j)$th entry is $\deg_{\kappa}(v_i \cdot w_j)$, i.e., using the degree map on $\overline{H}_{\kappa}(\Delta)$. 
By construction, $M_{\kappa}$ is invertible. Let $M$ be the matrix whose $(i, j)$th entry is $\deg(v_i \cdot w_j)$, i.e., using the degree map on $\overline{H}(\Delta)$. 
We obtain $M_{\kappa}$ from $M$ by reducing modulo $p$, see Remark~\ref{rem:specialize}, so $M$ is invertible. 
This implies that $\{v_1, \dotsc, v_r\}$ is linearly independent in $\overline{H}^m(\Delta)$, and so it is a basis for $\overline{H}^m(\Delta)$ by Corollary~\ref{cor:dimensioninvariant}. 

Let $\lambda \in {k}[a_{i,j}]$ be a polynomial such that, when we express $\lambda g$ in this basis, all coefficients are polynomials with coefficients in the valuation ring of ${k}$, and at least one of these polynomials has a coefficient whose image in $\kappa$ is nonzero. Let $\overline{\lambda g}$ denote the image of $\lambda g$ in $\overline{H}_{\kappa}(\Delta)$, which is nonzero. By Theorem~\ref{thm:4.4}, $\deg(\ell^{d - pm} \cdot  (\overline{\lambda g})^p)$ is nonzero. This implies that $\deg(\ell^{d - pm} \cdot (\lambda g)^p) = \lambda^p \deg(\ell^{d - pm} \cdot g^p)$ is nonzero, and so $\deg(\ell^{d - pm} \cdot g^p)$ is nonzero. 
\end{proof}

\begin{proof}[Proof of Theorem~\ref{thm:tpower}]
Let $p$ be a prime dividing $t$. By Proposition~\ref{prop:pansitropychar0}, $g^p$ is nonzero in $\overline{H}^{pm}(\Delta)$. Replacing $g$ by $g^p$ and continuing in this fashion gives the result. 
\end{proof}

\begin{remark}
The proof of Theorem~\ref{thm:tpower} works for any connected oriented simplicial $\mathbb{Q}$-homology manifold such that the homology of the link of any (possibly empty) face has no $p$-torsion for each prime $p$ dividing $t$. 
\end{remark}

\begin{example}\label{ex:RPd}
Let $\Delta$ be a triangulation of $\mathbb{R} \mathbb{P}^{d-1}$, where $d$ is even. Then $\Delta$ is a connected orientable simplicial $\Z$-homology manifold whose homology has no $3$-torsion. If $k$ has characteristic $0$ or $3$, then Proposition~\ref{prop:pansitropychar0} or Theorem~\ref{thm:4.4} implies that $\overline{H}(\Delta)$ has the weak Lefschetz property in degrees at most $\lfloor (d - 1)/3 \rfloor$, and so $\dim \overline{H}^0(\Delta) \le \dotsb \le \dim \overline{H}^{\lfloor (d - 1)/3 \rfloor}(\Delta)$. The results of \cite{PapadakisPetrotoug} or \cite{KaruXiaoAnisotropy} cannot be applied to such $\Delta$ because the dimension of $\overline{H}(\Delta)$ is different in characteristic $2$. 
\end{example}

\bibliography{pAnisotropy}
\bibliographystyle{amsalpha}

\end{document}